\documentclass{article}

\usepackage{placeins}
\usepackage{amsmath}
\usepackage{amsfonts}
\usepackage{tikz}
\usetikzlibrary{shapes}
\usepackage{graphicx}
\usetikzlibrary{cd}
\usetikzlibrary{babel}
\usepackage{amsthm}
\usepackage{authblk}

\usepackage{enumitem}
\usepackage{algorithm}
\usepackage{algpseudocode}
\usepackage{hyperref}
\newtheorem{definition}{Definition}
\newtheorem{theorem}{Theorem}

\usepackage[left=2cm, right=2cm, top=3cm]{geometry}

\newtheorem{proposition}{Proposition}

\usepackage{comment}

\title{Two approaches to low-parametric SimRank computation}
\author[1]{E. P. Berezin}
\author[2,3]{R. T. Zaks}
\author[2]{G. Z. Alekhin}
\author[3,4]{S. V. Morozov}
\author[2,3]{S. A. Matveev}
\affil[1]{Lomonosov Moscow State University Branch in the City of Sarov, Sarov, Russia, 607328}
\affil[2]{Lomonosov MSU, Faculty of Computational Mathematics and Cybernetics, Moscow, Russia, 119991}
\affil[3]{Marchuk Institute of Numerical Mathematics RAS, Moscow, 119333, Russia}
\affil[4]{National Research University Higher School of Economics, Moscow, 109028, Russia}

\date{}

\begin{document}
	
 \maketitle

 \begin{abstract} % You shouldn't use formulas and citations in the abstract.
In this work, we discuss low-parametric approaches for approximating SimRank matrices, which estimate the similarity between pairs of nodes in a graph. Although SimRank matrices and their computation require a significant amount of memory, common approaches mostly address the problem of algorithmic complexity. We propose two major formats for the economical embedding of target data. The first approach adopts a non‑symmetric form that can be computed using a specialized alternating optimization algorithm. The second is based on a symmetric representation and Newton-type iterations. We propose numerical implementations for both methodologies that avoid working with dense matrices and maintain low memory consumption. Furthermore, we study both types of embeddings numerically using real data from publicly available datasets. The results show that our algorithms yield a good approximation of the SimRank matrices, both in terms of the error norm (particularly the Chebyshev norm) and in preserving the average number of the most similar elements for each given node.
 \end{abstract}
% \subclass{97N40, 05C50, 05C82, 90B10, 65F99} % Enter 2010 Mathematics Subject Classification.
% \keywords{SimRank, low-rank approximations, quadratic optimization, Newton method, alternating minimization} % Include keywords separated by a comma.

 \section{Introduction}

 The problem of defining similarity between pairs of objects emerges in many applications, such as link prediction, graph clustering, and even molecular ecology \cite{DeSantis2011}. The SimRank is a similarity measure for graph data, proposed by Glen Jeh and Jennifer Widom at \cite{Jeh2002}. For a given graph $(V,E)$ with a set of vertices $V$ and set of edges $E$, each pair of vertices $(V_i, V_j)$ is associated  with a number $s(V_i, V_j) \in \left[0;1\right]$ according to relation 
\begin{equation}
s(V_i, V_j) =
\begin{cases}
1, \text{ if } V_i = V_j, \\
0, \text{ if } |I_i| = 0 \text{ or } |I_j| = 0, \\
\frac{c}{|I_i|\cdot|I_j|}\sum_{k \in I_i} \sum_{p \in I_j} s(V_k, V_p),
\end{cases} \label{eq:simrankorig}
\end{equation}
where $I_i$ is a set of edges incident to $V_i$ and $c$ is an arbitrary constant between 0 and 1. This constant defines the upper limit on similarity between two different vertices, $\max s(V_i,V_j) \leq c$ for $i \neq j$. In many practical computations, $c = 0.8$, and we also use this value in our paper. The intuition lying behind the SimRank is that two objects are similar if they are referenced by similar objects. 

The SimRank is rather well-studied nowadays, but it still has high computational and memory requirements for direct application to large graph-based data. One should note that model \eqref{eq:simrankorig} generates the linear fixed-point problem for the set of  $N(N-1)/2$ variables $s(V_i, V_j)$. The basic equations \eqref{eq:simrankorig} may be easily reformulated in a matrix form (see e.g. \cite{Kusumoto2014}) that is more feasible for organizing the computing process
\begin{equation*}
    S = (c~ A^{\top}SA) \lor I, \label{eq:simrankmatrix}
\end{equation*}
where $\{A \lor B\}_{ij} = \max(A_{ij},B_{ij})$. The matrix formulation can also be expressed as
\begin{equation}
    S = c~ A^{\top}SA + I - c~ \operatorname{diag}(A^{\top}SA), \label{eq:simrankmaineq}
\end{equation}
where $S \in \mathbb{R}^{n \times n} \colon S_{ij} = s(V_i, V_j)$, $S$ has important property: $S_{ii} = 1 ~~ \forall i=1,\dots,n$. $A \in \mathbb{R}^{n \times n}$ is a column-normalized graph adjacency matrix. Thus, $A^{\top}S$ is associated with $\frac{1}{|I_i|}\sum_{k \in I_i} s(V_k, V_p)$ and $SA$ with $\frac{1}{|I_j|} \sum_{p \in I_j} s(V_k, V_p)$ and $I$ is identity matrix of corresponding size. The operator $\operatorname{diag}$ simply takes the diagonal part of a matrix and nullifies its offdiagonal elements. 
We also utilize an operator for the off-diagonal part that will be useful in further calculations:
\begin{equation*}
\operatorname{off}(X) \equiv X - \operatorname{diag}(X).
\end{equation*}
Thus, the initial equation \eqref{eq:simrankmaineq} receives the compact form
\begin{equation}
    S = \underbrace{c~ \operatorname{off} (A^{\top}SA) + I}_{F(S)}. \label{eq:simrankmaineqoff}
\end{equation}
The simplest way to calculate the Simrank matrix is the fixed-point iteration method\footnote{The standard \texttt{simrank\_similarity} function of the popular Python NetworkX package exploits this approach.}  
\begin{equation*}
    S^{(k+1)} = F(S^{(k)}). \label{eq:FPoff}
\end{equation*}
It has some nice properties, such as convergence with linear rate for any nonnegative starting $S^{(0)}$. 

However, the SimRank model has two major drawbacks. The first one is computational complexity that is addressed in many papers. During the two decades since the original paper by Jeh and Widom has been published, several methods reducing the computational cost were born: (a) the pruning techniques \cite{Jeh2002}; (b) utilization of the sparse arithmetic, since adjacency matrix $A$ is usually sparse; (c) selecting the essential pairs of nodes and using partial sums \cite{Lizorkin2008}; (d) regularization and transition to relaxed equations, e.~g. the Sylvester equation \cite{Li2010} or the  discrete Lyapunov equation \cite{Yu2014} that unfortunately contains systematic errors \cite{Oseledets2016, Kusumoto2014}; (e) thresholding methods and GMRES to solve corresponding linear system with better convergence rate \cite{Oseledets2015} but consuming a lot of additional memory for storage of the Krylov bases; (f) the reversed Cuthill-McKee algorithm and the SOR method \cite{Lin2012}.

The second issue with SimRank is its memory requirements. The mentioned algorithms mostly do not deal with the problem of storage of dense $n \times n$ matrix $S$ containing the similarity scores, leading to unacceptable memory requirements in case of big data \cite{Yang2012}. In \cite{Oseledets2016}, an iterative method using the randomized low-rank spectral decomposition at each iteration has been proposed. It allows to reach $O(nr)$ memory cost (where $r$ is the rank of the corresponding decomposition) but struggles with relatively low accuracy of the final approximation.

This paper is organized as follows. We revisit the properties of the fixed-point iteration method in application to the Simrank problem in Section 2. Further, we concentrate on the low-rank parametrizations for the Simrank matrices. In Section 3, we test the low-rank parametrizations for the pre-computed solutions numerically and find that such decompositions exist and lead to high precision in terms of Chebyshev norm. We also discuss that our observations agree with existing theoretical results. In Section 4, we propose and test two possible ways for organizing the iterative process, finding the low-parametric approximation of the Simrank matrices. Finally, in Section 5, we demonstrate that novel algorithms outperform the method from \cite{Oseledets2016}. 

\section{Fixed-point iteration method}

In this section, we revisit the properties of the baseline fixed-point iteration method  \cite{Jeh2002} in application to the SimRank equation in a matrix form. For a given $A \in \mathbb{R}^{m \times n}$, Chebyshev norm $\|*\|_C$ is defined as $\|A\|_C = \max_{i,j} |A_{ij}|$. Remind that a column-stochastic matrix is a nonnegative matrix with sums of elements in its columns equal to 1. Note that $A$ can be considered as a result of the column normalization for the  binary adjacency matrix $\tilde{A}$.

\begin{proposition} \label{diagprop}

For any real $S \in \mathbb{R}^{n \times n}$ and column stochastic $A \in \mathbb{R}^{n \times n}$
\begin{equation*}
\| A^{\top}SA \|_C \leq \|S\|_C.
\end{equation*}

\begin{proof}
We denote $n_k$ as the number of ones in the $k$-th column of $\tilde{A}$, and $E$ is a matrix of ones. Then,
\begin{align*}
& \left| \{A^{\top}SA\}_{ij} \right| = 
\left| a_i^{\top}Sa_j \right|  =
 \frac{1}{ \|\tilde{a}_i\|_1 \|\tilde{a}_j\|_1 }  \tilde{a}_i^{\top}  |S|  \tilde{a}_j \leq
\frac{1}{ \|\tilde{a}_i\|_1 \|\tilde{a}_j\|_1 } \|S\|_C \tilde{a}_i^{\top} E \tilde{a}_j   \\ 
&  =
\frac{1}{ \|\tilde{a}_i\|_1 \|\tilde{a}_j\|_1 } \tilde{a}_i^{\top} \|S\|_C \begin{bmatrix} n_j & \dots & n_j \end{bmatrix}^{\top} =
\frac{1}{ \|\tilde{a}_i\|_1 \|\tilde{a}_j\|_1 } n_i n_j \|S\|_C = \|S\|_C,
\end{align*}
where $|S|$ denotes the matrix with absolute values of $S$.

\end{proof}
\end{proposition}

\begin{proposition}\label{prop2}
For $0 < c < 1$ and $S^{(k)}$ corresponding to the fixed-point iteration process with $F(S)$ we have
\begin{equation*}
|| F(S^{(k)}) - F(S^{(k-1)})||_C \leq c ||S^{(k)} - S^{(k-1)}||_C
\end{equation*}
for any iteration $k$.
\begin{proof}
Indeed, 
\begin{align*}
\| F(S^{(k)}) - F(S^{(k-1)}) \|_C &= \| S^{(k+1)} - S^{(k)} \|_C = \| c~ \operatorname{off} (A^{\top}S^{(k)}A) + I - c~ \operatorname{off} (A^{\top}S^{(k-1)}A) - I  \|_C \\
&= c~ \|\operatorname{off} (A^{\top}( S^{(k)} - S^{(k-1)} )A) \|_C \leq c~ \| A^{\top}( S^{(k)} - S^{(k-1)} )A \|_C.
\end{align*}

From Proposition \ref{diagprop} 
$
c~ \| A^{\top}( S^{(k)} - S^{(k-1)} )A \|_C \leq c~ \| S^{(k)} - S^{(k-1)} \|_C
$.

\end{proof}

\end{proposition}

\begin{proposition} [see also \cite{Lizorkin2008}] \label{LizorkinProposition1}
Let $S^{(k)}$ be the result of the fixed-point iterative process with $S^{(0)} = I$ and $S$ be exact solution of $F(S) = S$. Then $|| S - S^{(k)}||_C \leq  c^{k+1}$.

\begin{proof}
Directly follows from Propositions \ref{diagprop} and \ref{prop2}.
\end{proof}

\end{proposition}

Note that if $1 \geq S_{ij}=S_{jj} \geq 0$ then $F(S)$ keeps the output symmetric, with nonnegative elements bounded by 1. Hence, the solution of the Simrank problem via the fixed-iteration method exists and corresponds to nonnegative similarity scores if the initial matrix $S^{(0)}$ is nonnegative and symmetric.

\section{Inspiration for the low-rank approach}

Simrank allows capturing similar pairs of nodes that correspond to large elements in $S$. Thus, it is less important to approximate elements with small scores because they are not similar. Hence, our main interest lies in element-wise approximation of $S$. 

The singular values of $S$ for many datasets decay slowly (see Figure~\ref{fig:eumail_fb_singvals}), and straight-forward approximation of $S$ by the truncated singular value decomposition has poor accuracy. In Figure \ref{fig:eumail_fb_singvals} we present our approximations via truncated SVD for \texttt{email-Eu-core} dataset with 1005 nodes in graph \cite{eumailsource} and for \texttt{ego-Facebook} with 4039 nodes \cite{fbsource}. One may note that a simple shift $S - I$ allows one to obtain a better approximation as the corresponding singular values are consistently smaller in magnitude, but their decay is still very slow. 

At the same time, there are intriguing results guaranteeing the existence of the low-rank approximations in terms of Chebyshev norm (see \cite{Alon2013, Townsend2019, garnaev1984widths} for details).

\begin{figure}[htbp!]
\centering
\includegraphics[width = 0.5\textwidth]{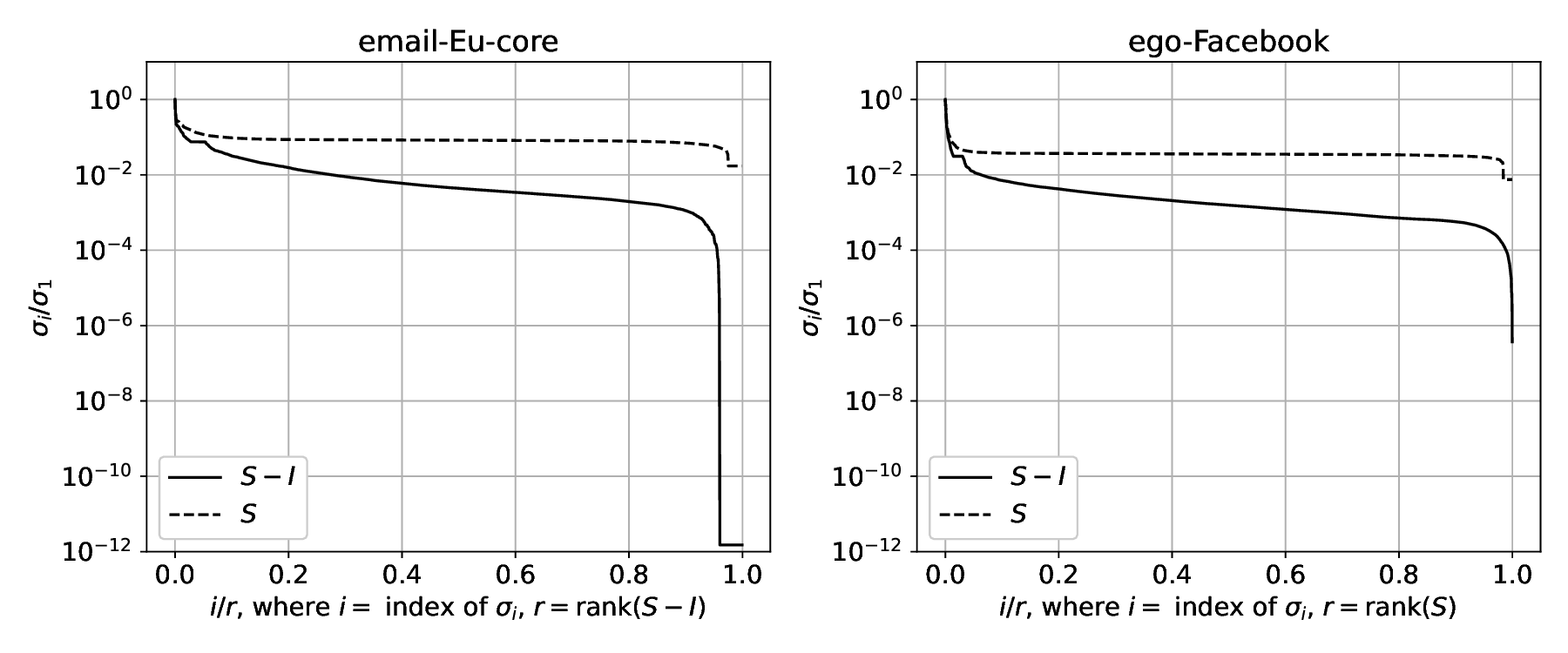}
\caption{Singular values of $S$ matrix and shift $S-I$ for \texttt{email-Eu-core dataset} with 1005 nodes in graph \cite{eumailsource} and for \texttt{ego-Facebook} with 4039 nodes \cite{fbsource} }
\label{fig:eumail_fb_singvals}
\end{figure}

\begin{figure}[htbp!]
\centering
\includegraphics[width = 0.5\textwidth]{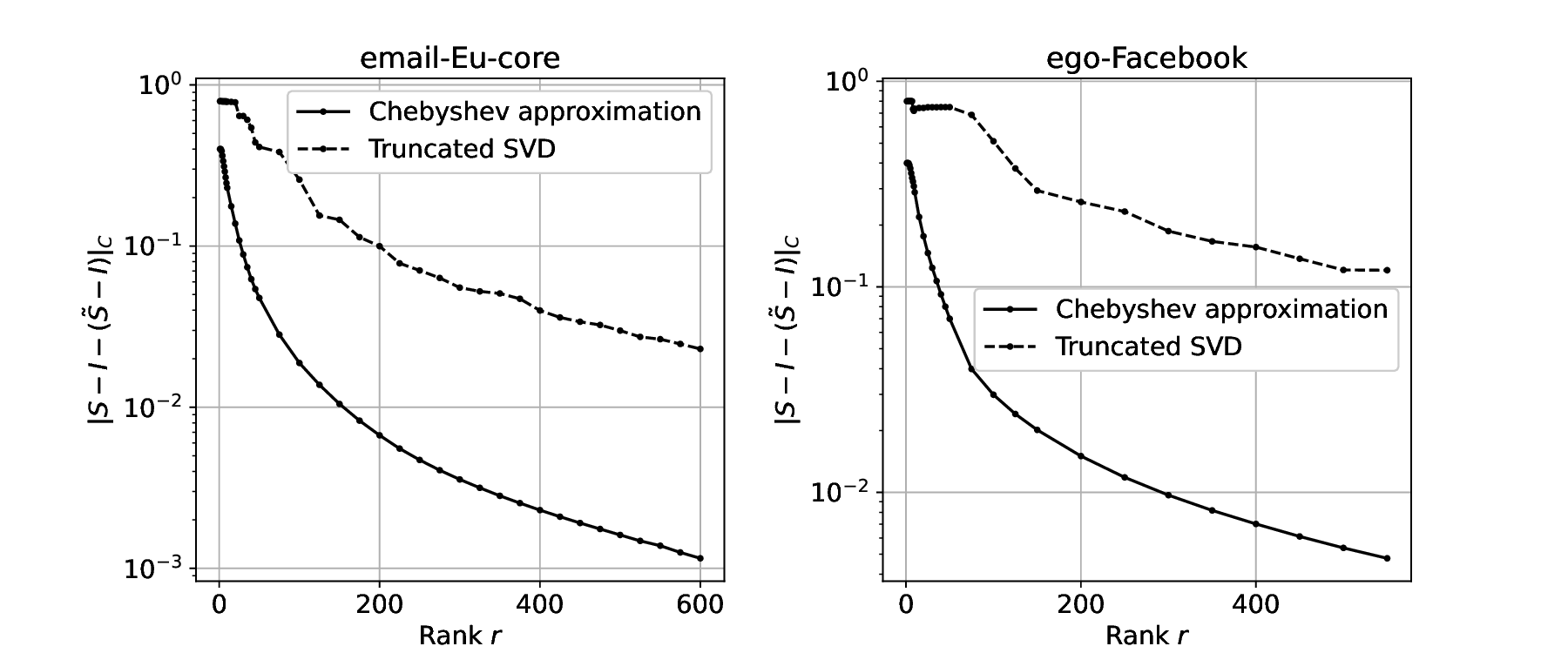}
\caption{Approximation error in Chebyshev norm for shift $S-I$ for \texttt{email-Eu-core} dataset with 1005 nodes in graph \cite{eumailsource} and for \texttt{ego-Facebook} with 4039 nodes \cite{fbsource} using truncated SVD and method from \cite{Zamarashkin2022} }
\label{fig:cheb_and_svd}
\end{figure}

\begin{definition}
Let $A \in \mathbb{R}^{n \times n}$, then
$
\varepsilon\operatorname{-rank}(A) = \min\{\operatorname{rank}(\tilde{A}) : \|A-\tilde{A}\|_C \leq \varepsilon \}.
$

\end{definition}

\begin{theorem}[Theorem 1.4 in \cite{Alon2013}] \label{AlonTheorem1}

Let $A \in \mathbb{R}^{n \times n}$, $A = A^{\top}$, positive semi-definite, $|A_{ij}| \leq 1$, $\forall i,j = 1, \dots,n$. Then $
\varepsilon\operatorname{-rank}(A) \leq \frac{9 \log n}{\varepsilon^2 - \varepsilon^3}.
$

\end{theorem}

Theorem \ref{AlonTheorem1}  allows us to expect the existence of the low-rank matrices approximating the SimRank matrix in the Chebyshev norm. 

\begin{proposition}
If $S^{(0)}$ is positive semi-definite matrix then $S^{(k)}$ from the fixed-point iteration method for $F(S)=S$ is also positive-semidefinite.
\begin{proof}
Let us consider the fixed-point iterations 
\begin{equation*}
S^{(k+1)} = c~ A^{\top}S^{(k)}A - c~ \operatorname{diag}(A^{\top}S^{(k)}A)+I.
\end{equation*}

$S^{(k)}$ is positive semi-definite $\Leftrightarrow $ $u^{\top}S^{(k)}u \geq 0$, $\forall$ $u \in \mathbb{R}^n$. Note, that
\begin{align*}
u^{\top}S^{(k+1)}u &= c~ u^{\top}A^{\top}S^{(k)}Au + u^{\top}(I- c~ \operatorname{diag}(A^{\top}S^{(k)}A))u \\
&= c(Au)^{\top}S^{(k)}(Au)+u^{\top} (I- c~ \operatorname{diag}(A^{\top}S^{(k)}A)) u.
\end{align*}

For the first term we note that $S^{(k)}$ is positive semi-definite, then $\forall u \in \mathbb{R}^{n}$ we have $(Au)^{\top}S^{(k)}(Au) \geq 0$. For the second term, we also see that
$$0 \leq \{ (I- c~ \operatorname{diag}(A^{\top}S^{(k)}A)) \}_{ij} \leq 1 \Rightarrow u^{\top} (I- c~ \operatorname{diag}(A^{\top}S^{(k)}A)) u \geq 0.$$
As far as both terms are nonnegative, we obtain
$
u^{\top}S^{(k+1)}u \geq 0
$.

\end{proof}
\end{proposition}

The specific numerical methods for approximation of matrices in Chebyshev norm are quite complicated and have high computational complexity \cite{Zamarashkin2022, Morozov2024, morozov2024refining}. We apply the alternating minimization method from \cite{morozov2024refining} for the relatively small Simrank matrices corresponding to several thousands of node pairs.

At first, we study the Simrank matrix for the \texttt{eumail-Eu-core} graph data with 1005 nodes whose singular values are depicted on the left panel of Figure \ref{fig:eumail_fb_singvals}. In this case, we obtain the approximations of rank 50 with accuracy higher than 0.1 using the chosen approach. At the same time, simple truncation of the singular value decomposition gives an error in Chebyshev norm of about 0.4 (as one can see in Figure \ref{fig:cheb_and_svd}). 

Further, we provide the same analysis for the Simrank matrix corresponding to the larger graph dataset containing 4039 nodes \cite{fbsource}. We present its singular values on the right panel of Figure \ref{fig:eumail_fb_singvals} and the accuracy of the low-rank approximations is shown on the right panel of Figure~\ref{fig:cheb_and_svd}. Our numerical observations agree with existing theory. 

Unfortunately, the approach from \cite{Morozov2024} can be utilized only for post-processing purposes. It also requires the storage of the whole matrix for organizing the calculations, which we would also like to avoid. Further, in this work, we propose heuristic iterative algorithms searching for acceptable low-rank parametrizations of the Simrank matrices from the very start of calculations.

\section{Two approaches to low-parametric solution}
\subsection{Alternating minimization}

Nowadays, there is a broad family of algorithms exploiting the idea of alternating minimization \cite{Morozov2024, morozov2024refining, Oseledets2023} and alternating projections \cite{Budzinskiy2025b, Budzinskiy2025c, Matveev2023} that are often being combined with miscellaneous heuristics. We approximate the solution of the Simrank problem as a low-parametric decomposition
\begin{equation}
\label{eq:s_represent}
\tilde{S} = I + UV^\top,
\end{equation}
where $\tilde{S}$ is approximation of $S$ and matrices $U, V \in \mathbb{R}^{n \times r}$ have to be found. Then, substituting \eqref{eq:s_represent} into \eqref{eq:simrankmaineqoff} and cancelling the identity on both sides yields
\begin{equation*}
UV^{\top} = c\,\operatorname{off}(A^{\top}UV^{\top}A) + c\,\operatorname{off}(A^{\top}A).
\end{equation*}
Define the constant matrix $B \equiv c\,\operatorname{off}(A^{\top}A)$,
so that the relation becomes
\begin{equation*}
UV^{\top} = c\,\operatorname{off}(A^{\top}UV^{\top}A) + B.
\end{equation*}

Let us define a linear operator $\mathcal{L}_V U: \mathbb{R}^{n \times r} \rightarrow \mathbb{R}^{n \times n}$ as $\mathcal{L}_V U \equiv UV^\top - c~ \operatorname{off} (A^{\top}UV^\top A)$,
where $V$ is fixed and operator $\mathcal{L}_U V : \mathbb{R}^{n \times r} \rightarrow \mathbb{R}^{n \times n}$ in similar manner but with fixed $U$. 

Following the ideology of alternating minimization, we propose to solve
\begin{equation}
||\mathcal{L}_V U-B||_F \rightarrow \underset{U}{\min} \label{eq:altmin_L_V}
\end{equation}
with some initial $V$ and obtain $\tilde{U}$. Then we should find $\tilde{V}$   
\begin{equation}
||\mathcal{L}_{\tilde{U}} V-B||_F \rightarrow \underset{V}{\min} \label{eq:altmin_L_U}
\end{equation}
using $\tilde{U}$ from previous step. All in all, we repeat this process solving \eqref{eq:altmin_L_V} and \eqref{eq:altmin_L_U} using the obtained $\tilde{U}$, $\tilde{V}$ as updates for $U$ and $V$ in the corresponding operators. Formally, the least squares problems \eqref{eq:altmin_L_V} and \eqref{eq:altmin_L_U}  can be solved exactly, but their algorithmic cost is unacceptable.

Setting $M = UV^{\top}$, we formulate the approximation as a minimization problem:
\begin{equation}
\label{eq:rank_r_minim}
\min_{\operatorname{rank} M=r} \| M - c\,\operatorname{off}(A^{\top} M A) - B \|_F.
\end{equation}
A natural fixed‑point iteration for \eqref{eq:rank_r_minim} is
\begin{equation*}
M^{(k+1)} = \operatorname{SVD}_{r}(c\,\operatorname{off}(A^{\top} M^{(k)} A) + B),
\end{equation*}
where $\operatorname{SVD}_{r}$ denotes the best rank‑$r$ approximation in the Frobenius norm. Computing a full SVD at every step can be expensive, especially for large $n$, and can be replaced with low-rank randomized spectral decomposition  according to \cite{Oseledets2016}. We present this trick in Algorithm \ref{alg:rsvd} and use it in our benchmarks.

We therefore adopt an alternating minimization strategy that avoids explicit SVD computation by keeping one of the low‑rank factors fixed. Suppose $U$ is held constant. With the current approximation $U (V^{(k)})^{\top}$, we seek a new $V^{(k+1)}$ that minimizes
\begin{equation*}
 \| U (V^{(k+1)})^{\top} - (c\,\operatorname{off}(A^{\top} U (V^{(k)})^{\top} A) + B) \|_F .
\end{equation*}
This linear least squares problem has a solution
\begin{equation*}
(V^{(k+1)})^{\top} = U^{\dagger}(c\,\operatorname{off}(A^{\top} U (V^{(k)})^{\top} A\big) + B),
\end{equation*}
where $U^{\dagger}$ denotes the Moore-Penrose pseudoinverse of $U$.
Similarly, if $V$ is fixed, we update $U$ by solving
\begin{equation*}
 \| V (U^{(k+1)})^{\top} - (c\,\operatorname{off}(A^{\top} V (U^{(k)})^{\top} A) + B^{\top}) \|_F \to \min_{U^{(k+1)}},
\end{equation*}
which gives
\begin{equation*}
(U^{(k+1)})^{\top} = V^{\dagger}(c\,\operatorname{off}(A^{\top} V (U^{(k)})^{\top} A\big) + B^{\top}).
\end{equation*}

Finally, we embed the alternating updates in an outer loop.  
Let $U^{(0,0)}, V^{(0,0)}\in\mathbb{R}^{n\times r}$ be initial guesses. For each outer iteration $m = 0,1,\dots,m_{\max}$ we perform a fixed number $k_{\max}$ of inner updates, first for $V$ while keeping $U$ fixed, then for $U$ while keeping the newly obtained $V$ fixed.

Set $U_{\text{fix}} = U^{(m,0)}$. For $k = 0,1,\dots,k_{\max}$ compute
\begin{equation*}
 (V^{(m,k+1)})^{\top} = U_{\text{fix}}^{\dagger}(c\,\operatorname{off}(A^{\top} U_{\text{fix}} (V^{(m,k)})^{\top} A\big) + B).
\end{equation*}
After the inner loop, set $V^{(m+1,0)} = V^{(m,k_{\max}+1)}$.

Then set $V_{\text{fix}} = V^{(m+1,0)}$. For $k = 0,1,\dots,k_{\max}$ compute
\begin{equation*}
 (U^{(m,k+1)})^{\top} = V_{\text{fix}}^{\dagger}(c\,\operatorname{off}(A^{\top} V_{\text{fix}} (U^{(m,k)})^{\top} A\big) + B^{\top}).
\end{equation*}
Finally, set $U^{(m+1,0)} = U^{(m,k_{\max}+1)}$. The process is repeated for increasing $m$ until $m = m_{\max}$. The final approximation of $S$ is
\begin{equation*}
 \tilde{S} = I + U^{(m_{\max} + 1,0)}(V^{(m_{\max} + 1,0)})^{\top}.
\end{equation*}
The detailed procedure is presented in Algorithm~\ref{alg:altopt}.

Despite this algorithm lacking a theoretical explanation for its convergence and approximation, it works surprisingly well within the studied datasets. It also has an important advantage of relatively low computational complexity, since the most expensive operation is pseudoinverse computation, but even this operation is not required at every iteration. 
\begin{algorithm}[H]
\caption{Alternating minimization}
\label{alg:altopt}

\textbf{Result}: matrices $U$, $V$

\begin{algorithmic}[1]
\State $U^{(0,0)}$, $V^{(0,0)}$ -- matrices with i.i.d. elements from the standard normal distribution
\State $B:= c~ \operatorname{off} (A^{\top}A)$
\For{$m = 0, \dots, m_{\text{max}}$}
 \State $U_{\text{fix}}:= U^{(m,0)}$
    \For{$k=0,\dots, k_{\text{max}}$}
        \State $(V^{(m,k+1)})^\top := {U_{\text{fix}}}^\dagger \left( c~ \operatorname{off} (A^{\top}U_{\text{fix}} (V^{(m,k)})^\top A) + B \right)$
    \EndFor
    \State $V^{(m+1,0)}:=V^{(m,k_{\text{max}}+1)}$, 
    $V_{\text{fix}} := V^{(m+1,0)}$
    \For{$k=0,\dots,k_{\text{max}} $}
        \State $(U^{(m,k+1)})^\top := {V_{\text{fix}}}^\dagger \left( c~ \operatorname{off} (A^{\top}V_{\text{fix}} (U^{(m,k)})^\top A) + B^\top \right)$
    \EndFor
    \State $U^{(m+1,0)}:=U^{(m,k_{\text{max}}+1)}$
    \If {Stoppingcriteria}
      \State $U := U^{(m+1,0)}$, $V := V^{(m+1,0)}$, Stop
    \EndIf
\EndFor
\end{algorithmic}
\end{algorithm}
From our numerical experiments, we see that taking initial matrices $U^{(0,0)}$, $V^{(0,0)}$ with random i.i.d elements from the standard normal distribution is crucial for obtaining a final good approximation. We obtain this fact only numerically and do not know any theoretical explanation of this phenomenon. We consider its theoretical justification as an interesting challenge for future research.

\subsection{Quadratic minimization}

Another approach utilizes the symmetric representation of $\tilde{S}$. We try to construct the solution as
\begin{equation}
\tilde{S} = I + \operatorname{off}(UU^{\top}). \label{eq:IoffUUT}
\end{equation}
This form also has $O(nr)$ memory cost. An interesting property of such data embedding is that the matrix $I + \operatorname{off}(UU^{\top})$ can actually be full-rank and still require less memory to store than a full $S$ matrix. One simple example is $U \in \mathbb{R}^{n}$:
\begin{equation*}
U = \begin{bmatrix} 0 & 0 & \dots & 0 & 1 & 2 \end{bmatrix}^{\top},
\end{equation*}
which leads to the full-rank matrix $S$. One may also easily find more complicated examples.

Hence, we call our decomposition low-parametric but not low-rank. Further, we form a minimization task allowing us to find the approximate solution of the SimRank problem in such a form. Substituting \eqref{eq:IoffUUT} into \eqref{eq:simrankmaineqoff},
\begin{equation*}
I + \operatorname{off}(UU^{\top}) = c~ \operatorname{off}(A^{\top}(I+\operatorname{off}(UU^{\top}))A)+ I,
\end{equation*}
\begin{equation*}
\operatorname{off}(UU^{\top}) = c~ \operatorname{off}(A^{\top} \operatorname{off}(UU^{\top}) A) + c~ \operatorname{off}(A^{\top}A).
\end{equation*}
We again denote $B \equiv c~ \operatorname{off}(A^{\top}A)$ and introduce the linear operator
\begin{equation*}
\Phi (X) \equiv \operatorname{off}(X) - c~ \operatorname{off}(A^{\top} \operatorname{off}(X) A).
\end{equation*}
%Now, our problem takes form
%\begin{equation}
%\Phi(UU^{\top})=B. \label{eq:Phi=B}
%\end{equation}
We introduce the residual functional $f:\mathbb{R}^{n \times r} \rightarrow \mathbb{R}$
\begin{equation*}
f(U) \equiv \| \Phi(UU^{\top}) - B \|_F^2,
\end{equation*}
and try to obtain the approximate low-parametric solution minimizing the value of $f(U)$. Thus, our problem becomes $f(U) \rightarrow \min_U$.

We call our approach  ``quadratic minimization'' because $f$ depends quadratically from $X=UU^{\top}$. Further, we set up an equation in order to solve our problem
\begin{equation}
\nabla f(U) = 0. \label{eq:nablaf=0}
\end{equation}
Since convexity analysis for $f(U)$ is quite complicated, we cannot guarantee to find global minima but hope to find at least a good local one. We solve \eqref{eq:nablaf=0} using the Newton method. The iterative process can be described by the following sequence
\begin{equation*}
U^{(k+1)} = U^{(k)} - J(U^{(k)})^{-1}[\nabla f(U^{(k)}]
\end{equation*}
where $J$ denotes the Jacobian operator of $\nabla f(U)$ (i.e., Hessian, which is actually a fourth-order tensor, but we avoid explicit construction of the tensor or its reshape to a matrix) and square brackets denote an argument of a linear operator. So $J(U^{(k)})^{-1}[\nabla f(U^{(k)})]$ means that inverse Jacobian is calculated at point $U^{(k)}$ and this operator is applied to $\nabla f(U^{(k)})$.

Instead of computing the derivatives component-wise, we use the framework of matrix derivatives. We denote $\langle X, Y \rangle \equiv \operatorname{tr} (Y^{\top}X)$ -- Frobenius inner product. Taking differential of $f(U)$ we see
\begin{align*}
df(U) 
&= d \langle \Phi (UU^{\top})-B , \Phi (UU^{\top})-B \rangle = 4 \langle \Phi^*(\Phi(UU^{\top})-B)U, dU \rangle
\end{align*}
where 
$\Phi^* (X) = \operatorname{off}(X) - c~ \operatorname{off}(A \operatorname{off}(X) A^{\top})$
is the adjoint of $\Phi$ in the Frobenius inner product.
It implies
\begin{equation}
\nabla f(U) = 4\Phi^*(\Phi(UU^{\top})-B)U. \label{eq:grad(f)}
\end{equation}
As we do not want to explicitly compute the Jacobian operator, we only seek the differential of the gradient: $J(U)[dU] = d(\nabla f(U))$. Then, applying the basic properties of differential, combined with the linearity of $\Phi$, we find that
\begin{align*}
d(\nabla f(U))
&= J(U)[dU] = d (4\Phi^*(\Phi(UU^{\top})-B)U) \\
&= 4 \Phi^*(\Phi( (dU)U^{\top} + U(dU)^{\top}))U + 4 \Phi^*(\Phi(UU^{\top})-B)dU.
\end{align*}
Summarizing our observation, we get
\begin{equation}
J(U)[X] = 4 \Phi^*(\Phi( XU^{\top} + UX^{\top}))U + 4 \Phi^*(\Phi(UU^{\top})-B)X. \label{eq:J(U)[dX]}
\end{equation}
Even though we seek to find the solution in a symmetric form, each step of the optimization takes more effort. The explicit Jacobian storage would require $O(n^2r^2)$ memory cells. Hence, we should solve the linear system
with Jacobian with internal iterative solver basing only on matrix by vector multiplications (e.~g. GMRES). We present a full set of operations for this method with Algorithm \ref{alg:quadmin_frame}.

\begin{algorithm}[H]
\caption{Quadratic minimization}
\label{alg:quadmin_frame}

\textbf{Result}: matrix $U$ such that $S \approx I + \operatorname{off}(UU^{\top})$

\begin{algorithmic}[1]
\State $U^{(0)}$ -- initial matrix, $k_{max}$ -- number of iterations
\State $B:= c~ \operatorname{off} (A^{\top}A)$
\State $f^{(0)} := f(U^{(0)}(U^{(0)})^{\top}) $
\For {$k = 0, \dots, k_{max}$}
 \State Solve $J(U^{(k)})[X] = \nabla f(U^{(k)})$
 \State $\Delta U^{(k+1)} := X$
 \State $U^{(k+1)} := U^{(k)} - \Delta U^{(k+1)}$ 
 \State $f^{(k+1)} := f(U^{(k+1)}(U^{(k+1)})^{\top})$
 \If {Stoppingcriteria}
      \State $U := U^{(k+1)}$
        \State Stop
    \EndIf
\EndFor
\end{algorithmic}
\end{algorithm}

Natural choice of the stopping criteria is to set some threshold for $f(U^{(k+1)})$ value. However, this approach requires to form full $\Phi(UU^{\top})$ matrix and violates our strategy to avoid $O(n^2)$ memory consumption. In the current study, we use this approach to control the iterations more precisely. One may try stopping the iterations if the shift of $\| \Delta U^{(k+1)} \|$ achieves a small value below the threshold. Steps 3 and 8 of an Algorithm \ref{alg:quadmin_frame} can be omitted if computation of $f$ is not necessary.

One can organize computations of the gradient \eqref{eq:grad(f)} and Jacobian operator \eqref{eq:J(U)[dX]} without forming the dense $n \times n$ matrices at any step of the algorithm (we assume that the normalized adjacency matrix $A$ is sparse), using the fact that for some matrices $X$,$Y$,$Z$ the value $\operatorname{diag}(XY)Z$ can be calculated effectively by multiplying $k$-th row of $Z$ by $x_{k,:}y_{:,k}$. Let us denote $\Phi^*(\Phi(XY))Z$ as $\operatorname{ffmp} (X,Y,Z)$ ($X$, $Y$, $Z$ -- some arbitrary matrices) and $\operatorname{diag}(XY)Z$ as $\operatorname{dmmp}(X,Y,Z)$. Note that both of the introduced $\operatorname{diag}(XY)Z$, $\operatorname{dmmp}(X,Y,Z)$ can be calculated effectively. Then we see 

\begin{equation*}
\Phi^*(\Phi(XY)) 
= \underbrace{\operatorname{off}(XY)- c~ \operatorname{off}(A^{\top} \operatorname{off}(XY)A)}_{T_1}
 - c~ \underbrace{\operatorname{off}(A \operatorname{off}(XY) A^{\top})}_{T_2}
 + c^2~  \underbrace{\operatorname{off}(A \operatorname{off}(A^{\top} \operatorname{off} (XY)A)A^{\top})}_{T_3}.
\end{equation*}
In more detail we look at the pieces of $\Phi^*(\Phi(XY))$ starting from $T_1$
\begin{multline*}
T_1Z = X(YZ) - \operatorname{dmmp}(X,Y,Z) - c~ (A^{\top}X)((YA)Z) + c~ A^{\top} \operatorname{dmmp} (X,Y,A)Z \\
 + c~ \operatorname{dmmp} (A^{\top}X, YA,Z) - c~ \operatorname{dmmp} (A^{\top}, \operatorname{dmmp} (X,Y,A),Z).
\end{multline*}
Then, we establish the equation for $T_2$
\begin{equation*}
T_2Z = (AX)((YA^{\top})Z) - A \operatorname{dmmp} (X, Y, A^{\top})Z - \operatorname{dmmp}(AX, YA^{\top}, Z)
 + \operatorname{dmmp}(A, \operatorname{dmmp} (X, Y, A^{\top}),Z),
\end{equation*}
and finally for $T_3$
\begin{multline*}
T_3Z
= (A(A^{\top}X))((YA)(A^{\top}Z)) 
 - A(A^{\top} ( \operatorname{dmmp} (X,Y,A)(A^{\top}Z))) 
 - A \operatorname{dmmp} (A^{\top}X, YA, A^{\top}Z) \\
 + A \operatorname{dmmp} (A^{\top}, \operatorname{dmmp}(X, Y, A), A^{\top}Z) 
 - \operatorname{dmmp} (A(A^{\top}X), (YA)A^{\top}, Z) \\
 + \operatorname{dmmp} (A, A^{\top} \operatorname{dmmp} (X,Y,A)A^{\top}, Z) 
 + \operatorname{dmmp} (A, \operatorname{dmmp} (A^{\top}X, YA, A^{\top}), Z) \\
 - \operatorname{dmmp} (A, \operatorname{dmmp} (A^{\top}, \operatorname{dmmp} (X,Y,A), A^{\top}, Z).
\end{multline*}
All in all, we denote
\begin{equation*}
\operatorname{ffmp} (X,Y,Z) = T_1 - c~ T_2 + c^2~  T_3.
\end{equation*}
In these notations, we obtain the equation for the gradient
\begin{equation*}
\nabla f(U) 
= 4  \Phi^*(\Phi(UU^{\top})-B)U
= 4 \left( \Phi^*(\Phi(UU^{\top})) - \Phi^*(B)U \right)
= 4 \left( \operatorname{ffmp} (U, U^{\top}, U) - \Phi^*(B)U \right)
\end{equation*}
and for the multiplication of the Jacobian matrix by a vector
\begin{multline*}
J(U)[X] 
= 4 \Phi^*(\Phi( XU^{\top} + UX^{\top}))U + 4 \Phi^*(\Phi(UU^{\top})-B)X \\
= 4 \left( \Phi^*(\Phi(XU^{\top}))U + \Phi^*(\Phi(UX^{\top}))U + \Phi^*(\Phi(UU^{\top}))X - \Phi^*(B)X \right) \\
= 4 \left( \operatorname{ffmp} (X,U^{\top},U) + \operatorname{ffmp} (U, X^{\top}, U) + \operatorname{ffmp} (U,U^{\top},X) - \Phi^*(B)X \right),
\end{multline*}
where $\Phi^*(B)$ is pre-computed. To solve $J(U^{(k)})[X] = \nabla f(U^{(k)})$ we use the GMRES solver. From our numerical experiments, we see the best results with a number of iterations between $10$ and $20$ within the studied datasets, and each dataset has its own optimal number of GMRES iterations. Detailed analysis of solver type and parameters that have an influence on the solution is a promising topic for further research.

\section{Numerical results}

In this section, we show our numerical results. We use datasets from the Stanford Large Network Dataset Collection (SNAP). All algorithms are implemented in Python 3 with NumPy and SciPy libraries. We are mainly interested in approximation in terms of the Chebyshev norm, but also introduce a special metric for the average number of the top-N correctly restored baseline similarities in terms of the initial Simrank model. Consider a graph with $n$ nodes (and correspondingly $A \in \mathbb{R}^{n \times n}, S \in \mathbb{R}^{n \times n}$). For a given $i$-th node,
\begin{equation*}
\psi_{i}(N) = \left| \mathcal{I}_{i}(N) \cap \mathcal{\tilde{I}}_{i}(N) \right|
\end{equation*}
where $\mathcal{I}_{i}(N)$ denotes a set of indices of top-N similar nodes for $i$-th node. $\mathcal{\tilde{I}}_{i}(N)$ denotes the same for matrix $\tilde{S}$ that approximates $S$.
Thus, we use
\begin{equation*}
\Psi(N) = \sum_{i=1}^n \frac{\psi_{i}(N)}{Nn}
\end{equation*} 
as a way to measure ``how much similar nodes are being conserved in approximation''. Such a function gives a more reliable way to measure the quality of parametrization since good approximation in terms of Chebyshev norm does not always means approximation of the similar elements in $S$ that are interesting for recommendations. We also test the alternative low-rank approach proposed in \cite{Oseledets2016} for the benchmark purpose (see Algorithm \ref{alg:rsvd}).

\begin{algorithm}[H]
\caption{Randomized SVD iterations}
\label{alg:rsvd}

\textbf{Result}: matrix $M = UU^\top$ such that $S \approx I + M$

\begin{algorithmic}[1]
\State $U^{(0)} \leftarrow \operatorname{SVD} \left( c~\operatorname{off} (A^\top A) \right) $
\For{$k = 0, \ldots ,k_{ \text{max}}$ }
    \State $\Omega \in \mathbb{R}^{n \times (r+p)}$ with i.i.d. elements from the standard normal distribution
    \State $M \Omega$ compute the tall sketch effectively 
    \State $QR \leftarrow M \Omega$ obtain thin QR decomposition
    \State $M := QQ^\top M_{\text{prev}}$ perform projection
    \State $U^{(k)} \leftarrow \operatorname{SVD} (Q^\top M_{\text{prev}})$
    \State $M_{\text{prev}} := M$
    \If {Stoppingcriteria}
        \State Stop
    \EndIf
\EndFor
\end{algorithmic}
\end{algorithm}

We present the benchmarks of the discussed approaches in Figure \ref{fig:chebtop_fb_wiki} for the \texttt{ego-Facebook} network dataset with 4039 nodes \cite{fbsource} (two plots on the left) and for \texttt{wiki-vote} data with 7115 nodes  (two plots on the right). We also test the alternating minimization for the high-energy physics citation network with 34546 nodes and present our results in Table~\ref{tab:hepph_article}. For all cases, we pre-calculate the full dense baseline Simrank solutions using the fixed-point iteration method with accuracy level $10^{-12}$ (in all cases, we exploit the sparsity of the corresponding adjacency matrices).

\begin{figure}[htbp!]
\centering
\includegraphics[width = 0.48\textwidth]{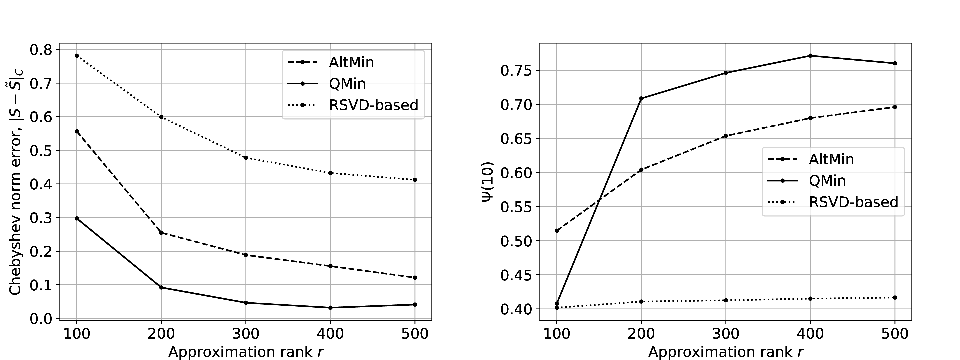}
\includegraphics[width = 0.48\textwidth]{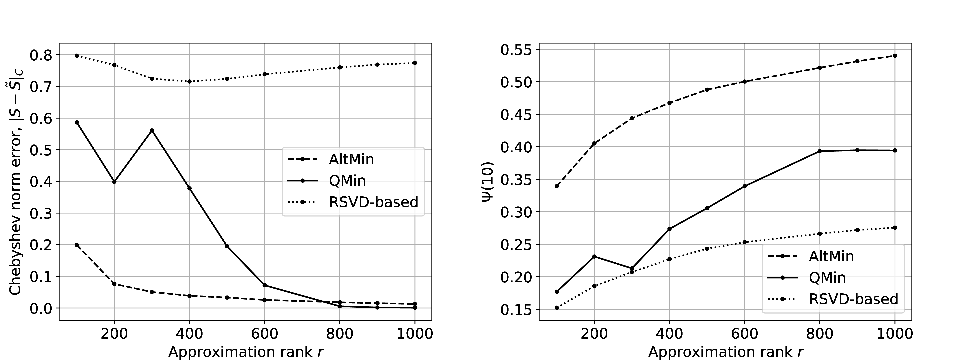}
\caption{Chebyshev norm error and $\Psi(10)$ for \texttt{ego-Facebook} dataset \cite{fbsource} (two plots on the left).
Chebyshev norm error and $\Psi(10)$ for \texttt{wiki-Vote} \cite{wikivotesource} (two plots on the right). }
\label{fig:chebtop_fb_wiki}
\end{figure}

At first, we see that novel alternating minimization and quadratic optimization methods outperform the RSVD-based method (see Algorithm~\ref{alg:rsvd}) in terms of correct top recommendations (the right panel of Figure~\ref{fig:chebtop_fb_wiki}). At the same time, we see that RSVD-based iterations do not approximate the solution in terms of Chebyshev norm at all, whether our approaches allow us to reach the error in Chebyshev norm less than 0.1 for rank values about 200. It means that the accurate numerical solution can be found in the low-parametric formats with a compression level of about 5 times. This level of compression is worse than offline (see Figure~\ref{fig:cheb_and_svd}) approximation with the method from \cite{morozov2024refining}, but the latter can be applied only to already computed full $n \times n$ Simrank matrix, which needs to be stored, and our methods allow to compute the solution directly in low-parametric form.

Further, we see that the alternating optimization outperforms the quadratic optimization for the larger dataset with 7115 nodes in terms of both error in Chebyshev norm and recommendations accuracy $\Psi(10)$. Also, we note that the error in Chebyshev norm and $\Psi(10)$ are perfectly correlated. However, for matrices with many large elements that are close to each other in terms of the basic Simrank equations, the relatively small error around 0.1 in terms of Chebyshev norm for the rank 200 gives quite a poor value of $\Psi(10)$. We finally break through the level of 50\% top-10 recommendations per each user, only for the value of rank about 800, because it contains many large elements that are close to each other in terms of the Simrank model. Anyway, even such a value of rank allows obtaining the approximate solution with a compression by more than 4.4 times.

\begin{table}
\centering
\begin{tabular}{c|c|c|c|c|c|c|c}
\hline
rank & 2000 &  2250 & 2500 & 2750 & 3000 & 3250 & 3500 \\
\hline 
$||S - \tilde{S}||_C$ & 0.59881289 & 0.55076561 & 0.44502145 & 0.4388728 & 0.41307087 & 0.40972939 & 0.4  \\
\hline
$\Psi(10)$ & 0.5053436 & 0.53758178 & 0.55900828 & 0.57520408 & 0.58949227 & 0.60304522 & 0.61528976 \\
\hline
\end{tabular}
\caption{Chebyshev norm error and $\Psi(10)$ for the high-energy physics citation graph \cite{snapnets}.}
\label{tab:hepph_article}
\end{table}

We present the last test for the high-energy physics citation graph \cite{snapnets} containing 34546 nodes with Table~\ref{tab:hepph_article} only for the alternating minimization method, because other approaches take too much time. We reach $\Psi(10) \approx 0.6$ for setting 3250 as the value of the rank, meaning that the approximation of the Simrank matrix can reach the compression level by more than 5.3 times.

\section{Conclusion and future work}

In this work, we propose two low-parametric approaches for the storage and computation of the SimRank matrices. We show that similarity matrices can be well-approximated in terms of the Chebyshev norm. We present benchmarks of our algorithms on real datasets from the open SNAP collection and observe that SimRank matrices can be approximated in the compressed form with significantly fewer parameters from the very beginning of the iterative process.

Although our results are promising for both novel methods, there remain numerous directions for further research: (a) both approaches require additional theoretical justification, (b) the computational complexity should be reduced, (c) possible convergence criteria need to be investigated in more detail, and (d) parallel structure should also be studied.

\section*{Funding}

Sergey Matveev was supported by Russian Science Foundation grant 25-11-00392 (\url{https://rscf.ru/project/25-11-00392/}). Sergey Matveev is grateful to Stanislav Budzinskiy for fruitful discussions on theory about low-rank approximations in Chebyshev norm. Our codes are openly available at \url{https://github.com/ep0berezin/Simrank_LPM}.

\bibliographystyle{plainurl}
\bibliography{bibliography}

\end{document}